\documentclass{article}
\usepackage{amsmath}
\usepackage{amssymb}
\usepackage{amsthm}
\usepackage[utf8]{inputenc}
\usepackage[shortlabels]{enumitem}

\newcommand{\bi}{\mathbf{i}}
\newcommand{\bx}{\mathbf{x}}

\newtheorem{theorem}{Theorem}[section]

\newtheorem{definition}[theorem]{Definition}
\newtheorem{prop}[theorem]{Proposition}

\allowdisplaybreaks

\title{A Generalization of the Graham-Pollak Tree Theorem to Even-Order Steiner Distance}
\author{Joshua Cooper and Gabrielle Tauscheck}
\date{\today}

\begin{document}

\maketitle

\begin{abstract}
Graham and Pollak showed in 1971 that the determinant of a tree's distance matrix depends only on its number of vertices, and, in particular, it is always nonzero.  The Steiner distance of a collection of $k$ vertices in a graph is the fewest number of edges in any connected subgraph containing those vertices; for $k=2$, this reduces to the ordinary definition of graphical distance.  Here, we show that the hyperdeterminant of the $k$-th order Steiner distance hypermatrix is always nonzero if $k$ is even, extending their result beyond $k=2$.
Previously, the authors showed that the $k$-Steiner distance hyperdeterminant is always zero for $k$ odd, so together this provides a generalization to all $k$.  We conjecture that not just the vanishing, but the value itself, of the $k$-Steiner distance hyperdeterminant of an $n$-vertex tree depends only on $k$ and $n$. 
\end{abstract}

\section{Introduction}

In an influential 1971 paper, Graham and Pollak (\cite{GrPo71}) computed the determinant of the distance matrix of a $n$-vertex tree, i.e., the $n \times n$ matrix $D$ whose $(i,j)$ entry is the shortest-path distance between $i$ and $j$.  Strikingly, the formula only depends on $n$, and nothing more about the tree:
\begin{equation} \label{eq-1}
\det(D) = (1 - n)(-2)^{n-2}.
\end{equation}
In particular, this quantity is nonzero for any $n \geq 1$.

Steiner distance generalizes the classical notion of distance in a graph $G$ from pairs of vertices to any subset $S \subseteq V(G)$, as the minimum number of edges in any connected subgraph of $G$ containing $S$ (or possibly $\infty$ if there is no such subgraph).  Therefore, partly inspired by recent developments in hypermatrices, Y.~Mao asks (q.v.~\cite{Mao17}) how to generalize (\ref{eq-1}) to hyperdeterminants of order-$k$ Steiner distances in trees.  The present authors showed previously (\cite{CoTa24}) that, for $k$ odd (and $n \geq 3$), the Steiner distance hyperdeterminant is $0$.  In this paper, we show that the Steiner distance is always {\em nonzero} for $k$ even (and $n \geq 2$), implying a weak generalization of the Graham-Pollak Theorem to Steiner distance: for $n \geq 3$, whether the hyperdeterminant vanishes depends only on $k$.  In fact, we conjecture a much stronger statement: the value of the hyperdeterminant (not just whether it vanishes) only depends on $n$ and $k$ in general.  This has been checked computationally for all trees with $(k,n) \in \{(4,4),(4,5),(6,4)\}$ and holds trivially for $n=2$, $n=3$.  The following table shows the common values of the hyperdeterminant, factored into primes.
$$
\begin{array}{l|l}
(n,k) & \det(\mathcal{S}) \\
\hline
\hline
(4,2) & -2^2 \cdot 7 \\
(4,3) & 2^{12} \cdot 7 \cdot 23^4 \\
(4,4) & - 2^{38} \cdot 3^{27} \cdot 5^6 \cdot 7 \cdot 13^{12} \\
(4,5) & 2^{203} \cdot 5^{32} \cdot 7 \cdot 11^{32} \cdot 23^{24} \cdot 37^8 \\
(6,2) & - 11^2 \cdot 31 \\
(6,3) & 2^{14} \cdot 3^{16} \cdot 11^4 \cdot 31 \cdot 19231^4 \\
(6,4) & - 2^{82} \cdot 3^{17} \cdot 11^8 \cdot 31 \cdot 41^{12} \cdot 71^6 \cdot 89^6 \cdot 151^{24} \cdot 257^{24} \cdot 1511^{12} \\
(8,2) & - 2^6 \cdot 29^2 \cdot 127 \\
(8,3) & 2^{56} \cdot 13^{16} \cdot 29^4 \cdot 113^8 \cdot 127 \cdot 1009^8 \cdot 2143^4 
\end{array}
$$

First, some notation and facts about hyperdeterminants are needed.  The following theorem is a useful generalization of the fact that the matrix determinant is the unique irreducible polynomial in the entries of $M$ whose vanishing coincides with the existence of nontrivial solutions to the linear system corresponding to $M$.

\begin{theorem}[\cite{GelKapZel08} Theorem 1.3] \label{thm:hyperdeterminant} The hyperdeterminant $\det(M)$ of the order-$k$, dimension-$n$ hypermatrix $M = (M_{i_1,\ldots,i_k})_{i_1,\ldots,i_k=1}^n$ is a monic irreducible polynomial which evaluates to zero iff there is a nonzero simultaneous solution to $\nabla f_M = \vec{0}$, where
$$
f_M(x_1,\ldots,x_n) = \sum_{i_1,\ldots,i_k} M_{i_1,\ldots,i_k} \prod_{j=1}^k x_{i_j}.
$$
\end{theorem}

Next, we define Steiner distance.

\begin{definition} Given a graph $G$ and a subset $S$ of the vertices, the {\em Steiner distance} of $S$, written $d_G(S)$ or $d_G(v_1,\ldots,v_k)$ where $S = \{v_1,\ldots,v_k\}$, is the number of edges in the smallest connected subgraph of $G$ containing $S=\{v_1,\ldots,v_k\}$. We often supress the variable $v$ when referring to the vertices within $S$. Since such a connected subgraph of $G$ witnessing $d_G(S)$ is necessarily a tree, it is called a {\em Steiner tree} of $S$. 
\end{definition}

Then, we have a ``Steiner polynomial'' which reduces to the quadratic form defined by the distance matrix for $k=2$.

\begin{definition} Given a graph $G$, the {\em Steiner polynomial} of $G$ is the $k$-form 
$$
p^{(k)}_G(\mathbf{x}) = \sum_{v_1,\ldots,v_k \in V(G)} d_G(v_1,\ldots,v_k) x_1 \cdots x_k
$$
where we often suppress the subscript and/or superscript on $p^{(k)}_G$ if it is clear from context.
\end{definition}

Equivalently, we could define the Steiner polynomial to be the $k$-form associated with the Steiner hypermatrix:

\begin{definition} Given a graph $G$, the {\em Steiner $k$-matrix} (or just ``Steiner hypermatrix'' if $k$ is understood) of $G$ is the order-$k$, cubical hypermatrix $\mathcal{S}_G$ of dimension $n$ whose $(v_1,\ldots,v_k)$ entry is $d_G(v_1,\ldots,v_k)$.
\end{definition}

Throughout the sequel, we write $D_r$ for the operator $\partial / \partial x_r$, and we always assume that $T$ is a tree.

\begin{definition} Given a graph $G$ on $n$ vertices, the {\em Steiner $k$-ideal} -- or just ``Steiner ideal'' if $k$ is clear -- of $G$ is the ideal in $\mathbb{C}[x_1,\ldots,x_n]$ generated by the polynomials $\{D_j p_G \}_{j=1}^n$.  
\end{definition}

Thus, the Steiner ideal is the {\em Jacobian ideal} of the Steiner polynomial of $G$. Equivalently, we could define the Steiner $k$-ideal via hypermatrices. We use the notation $\mathcal{S}(x^{k-1},*)$ to represent $\mathcal{S}(\overbrace{x,x,\ldots, x}^{k-1},*).$ Note that $S(x^{k-1},*) = S(x^{k-2},*,*)x$ where $S(x^{k-2},*,*)$ is an $n\times n$ matrix whose entries are homogeneous polynomials of degree $k-2$. Then the {\em Steiner $k$-ideal} of a graph $G$ on $n$ vertices is the ideal in $\mathbb{C}[x_1, \ldots, x_n]$ generated by the components of $\mathcal{S}(x^{k-1},*)$. 

\begin{definition} A {\em Steiner nullvector} is a point where all the polynomials within the Steiner ideal vanish. The set of all Steiner nullvectors -- a projective variety -- is the {\em Steiner nullvariety}.
\end{definition}

\section{Main Results}
We show that the order-$k$ Steiner distance hypermatrix of a tree $T$ on $n + 1 \geq 2$ vertices has a nonzero hyperdeterminant for even $k$. Therefore, for the remainder of this paper, we assume $T$ is a tree on at least $2$ vertices and $k$ is even. For concision, let $S = \mathcal{S}(x^{k-2},*,*)$ as defined in the introduction, and write $S_u$ for the row of $S$ indexed by $u$ and $S_{u,v}$ for the $v$ entry of $S_u$.  We also employ multi-index notation: for $\bi = (i_1,\ldots,i_r) \in \{0,\ldots, n\}^r$ and $\bx = [x_0,\ldots,x_n]$, write $\bx^{\bi}:= \prod_{j =1}^{r} x_{i_j}$ and $\nu_t(\bi) := \vert\{j:i_j = t\}\vert.$ 

\begin{prop} \label{Arows}
If $e=\{u,u+1\}$ is an edge of the tree $T$ on $n+1$ vertices $\{0,\ldots,n\}$, then $S_u - S_{u+1}$ is a vector of the form $\bx = [x_0, x_1, \ldots, x_n]$ where $x_0 = x_1 = \ldots = x_u = -\left(\sum_{i=0}^ux_i\right)^{k-2}$ and $x_{u+1} = \ldots = x_n = \left(\sum_{i = u+1}^nx_i\right)^{k-2}$ with $0, \ldots, u$ lying in one component of $T' := T - e$ and $u+1, \ldots, n$ lying in the other component. 
\end{prop}

\begin{proof}
     Note that 
     \begin{equation} \label{eq0}
     S_{uw} = \sum_{\bi \in V(T)^{k-2}} d_T(u,w,\bi) \bx^\bi.
     \end{equation} 
     It is sufficient to show that for any $w$ and $w'$ that lie within the same component of $T'$, 
     $$
     d_T(u, w, \bi) - d_T({u+1}, w, \bi) = d_T(u, {w'}, \bi) - d_T({u+1}, {w'}, \bi)
     $$ 
     holds for all $i_1,\ldots,i_{k-2} \in V(T)$. For every case, let $d = d_T(u, w, \bi)$ and $d' = d_T(u, {w'}, \bi)$.

     Case 1: Let $i_1,\ldots,i_{k-2}$ lie in the same component of $T'$ as $u$. 
     \begin{enumerate}[(a)]
        \item Assume $w$ and $w'$ also lie within the same component as $u$. Then, 
            $$
            d_T(u, w, \bi) - d_T({u+1}, w, \bi) = d - (d+1) = -1
            $$
            $$
            d_T(u, {w'}, \bi) - d_T({u+1}, {w'}, \bi) = d' - (d' + 1) = -1.
            $$
        \item Assume instead that $w$ and $w'$ lie in the same component as $u+1$. Then,
        $$
        d_T(u, w, \bi) - d_T(u+1, w, \bi) = d - d = 0 
        $$
        $$
        d_T(u, w', \bi) - d_T(u+1, w', \bi)= d' - d' = 0.
        $$
    \end{enumerate}

     Case 2: Let $i_1,\ldots,i_{k-2}$ lie in the same component as $u+1$. 
     \begin{enumerate}[(a)]
         \item Assume $w$ and $w'$ also lie within the same component as $u+1$. Then, 
         $$
         d_T(u, w, \bi) - d_T(u+1, w, \bi) = d - (d-1) = 1 
         $$
         $$
         d_T(u, w', \bi) - d_T(u+1, w', \bi) =d' - (d' - 1) = 1.
         $$

         \item Assume instead that $w$ and $w'$ lie in the same component as $u$. Then, 
         $$
         d_T(u, w, \bi) - d_T(u+1, w, \bi) = d - d = 0 
         $$
         $$
         d_T(u, w', \bi) - d_T(u+1, w', \bi)= d' - d' =0.
         $$
     \end{enumerate}

     Case 3: Suppose some of the vertices among $i_1,\ldots,i_{k-2}$ lie in the same component as $u$ and some lie in the same component as $u+1$. Then, 
     $$
     d_T(u, w, \bi) - d_T(u+1, w, \bi) = d - d = 0
     $$
     $$
     d_T(u, w', \bi) - d_T(u+1, w', \bi)= d' - d'=0.
     $$

     Therefore, the vector $S_u - S_{u+1}$ has equal entries among the coordinates indexed by vertices lying in the same component of $T'$.  To compute these entries, we combine the different cases. Notice that Case 3 contributes zero to the sum in (\ref{eq0}); therefore, we can omit those entries in the calculations. First, consider $x_u = S_{u,u}-S_{u+1,u}$. Since $u$ lies in the same component as itself, we are working with Cases 1(a) and 2(b). Case 2(b) contributes a zero to the summation; therefore, we need only consider the vertices in Case 1(a), giving us $x_u = -\sum_{\bi \in \{0,\ldots,u\}^{k-2}} \bx^\bi = -\left(\sum_{i=0}^ux_i\right)^{k-2}$. 

     Now, consider $x_{u+1} = S_{u,u+1} - S_{u+1,u+1}$. Since $u+1$ lies in the same component as itself, we are working with Cases 1(b) and 2(a). Case 1(b) contributes a zero to the summation; therefore, we need only consider the vertices in Case 2(a), giving us $x_{u+1} = \sum_{\bi \in \{u+1,\ldots,n\}^{k-2}} \bx^\bi = \left(\sum_{i=u+1}^nx_i\right)^{k-2}.$
\end{proof}

Let $\textbf{h}_u(a,b) = [h_0, h_1,\ldots, h_n] \in \mathbb{C}^{n+1}$ be a vector with at most two distinct entries: $h_0 = \ldots = h_u = a$ and $h_{u+1} = \ldots = h_n = b$.

\begin{prop}\label{leaves}
Let $\bx = [x_0, x_1, \ldots, x_n]$ be a Steiner nullvector of a tree $T$. All entries that correspond to leaf vertices $v$ are the same value, equal to the summation of every entry of $\bx$ except $x_v$. 
\end{prop}

\begin{proof}
        Since $T$ is a tree on at least $2$ vertices, it must have a leaf vertex; call it $0$, and let $1$ be its unique neighbor. Proposition \ref{Arows} gives $a$ and $b$ so that \begin{align*}
            (D_0p_T - D_1p_T)(\bx) &= \textbf{h}_0(a,b)^T \bx\\
            &= \textbf{h}_0\left(-x_0^{k-2}, \left(\sum_{i=1}^nx_i\right)^{k-2}\right)^T \bx\\
            &= -x_0^{k-1} + \left(\sum_{i=1}^nx_i\right)^{k-1} = 0
        \end{align*}
    so we may conclude that $x_0^{k-1} = \left(\sum_{i=1}^nx_i\right)^{k-1}$. Since $k-1$ is odd, 
    $$
    x_0 = \sum_{i=1}^nx_i.
    $$ 
    
    From here, we see that $\sum_{i=0}^nx_i = x_0 + x_0$ so that $x_0 = \frac{1}{2}\sum_{i=0}^nx_i$. The choice of vertex $0$ was made arbitrarily among the leaves; therefore, every leaf corresponds to this same entry in the nullvector, namely, $\sum_{i=1}^n x_i$.  
\end{proof}

\begin{prop} \label{summations}
    Suppose $e = \{u,u+1\}$ is an edge of the tree $T$ on $n+1$ vertices. Let $0, \ldots, u$ lie in one component of $T - e$ and $u+1, \ldots, n$ lie in the other component. Let $\bx = [x_0, x_1, \ldots, x_n]$ be a Steiner nullvector. Then $\sum_{i=0}^ux_i = x_v$ where $v$ is any leaf vertex. 
\end{prop}

    \begin{proof}
    Proposition \ref{Arows} gives $a$ and $b$ so that 
    \begin{align*}
        (D_up_T - D_{u+1}p_T)(\bx) &= \textbf{h}_u(a,b)^T \bx\\
        &= \textbf{h}_u\left(-\left(\sum_{i=0}^ux_i\right)^{k-2}, \left(\sum_{i=u+1}^nx_i\right)^{k-2}\right)^T \bx\\
        &= -\left(\sum_{i=0}^ux_i\right)^{k-1} + \left(\sum_{i=u+1}^nx_i\right)^{k-1} = 0
    \end{align*}
    so we may conclude that $\left(\sum_{i=0}^ux_i\right)^{k-1} = \left(\sum_{i=u+1}^nx_i\right)^{k-1}$. Since $k-1$ is odd,
    \begin{equation}\label{eq1}
    \sum_{i=0}^ux_i = \sum_{i=u+1}^nx_i.
    \end{equation}
    
    Assume without loss of generality that $0$ is a leaf vertex; then Proposition \ref{leaves} gives that $x_0 = \sum_{i=1}^nx_i$. We can rewrite (\ref{eq1}) to be $x_0 + \sum_{i=1}^ux_i = x_0 - \sum_{i=1}^ux_i$ so that $$\sum_{i=1}^ux_i = 0.$$
    Therefore, $\sum_{i=0}^ux_i = x_0$ as desired. 
    \end{proof}

    We are now in a position to describe all entries of a Steiner nullvector of a tree $T$. 

\begin{prop}\label{interior}
    Suppose $\bx = [x_0,\ldots,x_n]$ is a Steiner nullvector of a tree $T$ on $n+1$ vertices.  Then, for each $t \in V(T)$, $x_t = x_v(2-\deg(t))$ where $v$ is any leaf vertex. 
\end{prop}

\begin{proof}
    First fix a leaf vertex $v$; note that Proposition \ref{leaves} implies that $x_v$ does not depend on which leaf is chosen. Since $\deg(v) = 1$, we have $x_v = x_v(2-\deg(v))$. 
    
    Let $t$ be any nonleaf vertex, with degree $d > 1$.  Since $T$ is a tree on at least $2$ vertices, $t$ has at least one neighbor -- call it $u$ -- and let $e$ be the edge $\{t,u\}$.  Let $\mathcal{G}$ be the set of all vertices that lie in the same component as $t$ in $T - e$. Proposition \ref{summations} tells us that $\sum_{i \in \mathcal{G}}x_i = x_v$. Partition $\mathcal{G}$ into sets $\{t\},\mathcal{G}_0, \ldots, \mathcal{G}_{d-2}$ by taking $\mathcal{G}_j$ to be vertices of the $j$-th component induced by $\mathcal{G} \setminus \{t\}$. (Note that there are only $d-1$ subsets since the component of $T - \{t\}$ containing vertex $t+1$ is not included in $\mathcal{G}$.) Once again, we can apply Proposition \ref{summations} to say that $\sum_{i \in \mathcal{G}_j} x_i = x_v$ for $0 \leq j \leq d-2$. Therefore, 
    \begin{align*}
        x_v &= \sum_{i \in \mathcal{G}} x_i \\
        &= x_t + \sum_{j=0}^{d-2} \sum_{i \in \mathcal{G}_j} x_i \\
        &= x_t + \sum_{j=0}^{d-2}x_v \\
        &= x_t + x_v(\deg(t)-1)
    \end{align*}
    and $x_t = x_v(2-\deg(t))$ as desired. 
\end{proof}

We are now equipped to prove the main result. 

\begin{theorem}\label{trivial}
    For $k$ even, the Steiner $k$-matrix of a tree on at least $2$ vertices has a nonzero hyperdeterminant. 
\end{theorem}

\begin{proof}
    Using Proposition \ref{interior}, we prove the result by contradiction. 
    
    Consider any tree $T$ on $n+1$ vertices where $n \geq 1$. Assume without loss of generality that $0$ is a leaf vertex of $T$ whose neighbor is $1$. Assume the hyperdeterminant is zero. Theorem \ref{thm:hyperdeterminant} then implies that there exists a nontrivial common zero of the Steiner $k$-ideal; let $\bx = [x_0,x_1, \ldots, x_n]$ be a Steiner nullvector of $T$.  Since the generators of the Steiner $k$-ideal are homogeneous, and $x_0=0$ implies that $x_v=0$ for all $v$ by Proposition \ref{interior}, we may assume $x_0 = 1$. Then,
    \begin{align*}
    D_1p_T(\bx) &= \sum_{\bi\in \{0,\ldots, n\}^{k-1}}d_T(1, \bi)\bx^{\bi}\\
    &= \sum_{a=0}^{k-1}\sum_{\substack{\bi \in \{0,\ldots, n\}^{k-1} \\ \nu_0(\bi) = a}}d_T(1,\bi,\overbrace{0,0,\ldots, 0}^{a})\bx^{\bi}\\
    &= \sum_{a=0}^{k-1}{k-1\choose a} x_0^a\sum_{\bi \in \{1,\ldots, n\}^{k-1-a}}d_T(1,\bi,0)\bx^{\bi}\\
    &= \sum_{\bi \in \{1,\ldots, n\}^{k-1}}d_T(1, \bi)\bx^{\bi}+\sum_{a=1}^{k-1}{k-1 \choose a}\sum_{\bi \in \{1,\ldots, n\}^{k-1-a}}[d_T(1,\bi)+1]\bx^{\bi}
\end{align*}
by partitioning the vectors $\bi$ according whether $\nu_0(\bi) = 0$ or $a$ for some $a>0$.  Continuing to rewrite the quantity $D_1p_T(\bx)$,
\begin{align*}
    &= \sum_{a=0}^{k-1}{k-1\choose a} \sum_{\bi \in \{1,\ldots, n\}^{k-1-a}}d_T(1,\bi)\bx^{\bi}+\sum_{a=1}^{k-1}{k-1\choose a}\sum_{\bi \in \{1,\ldots, n\}^{k-1-a}}\bx^{\bi}\\
    &= \sum_{a=0}^{k-1}{k-1\choose a}\sum_{\bi \in \{1,\ldots, n\}^{k-1-a}}d_T(1,\bi)\bx^{\bi} + \left(1+\sum_{i=1}^{n}x_i\right)^{k-1} - \left(\sum_{i=1}^{n}x_i\right)^{k-1}\\
    &= \left(\sum_{a=0}^{k-1}{k-1\choose a} \sum_{b=0}^{k-1-a}\sum_{\substack{\bi\in \{1,\ldots, n\}^{k-1-a}\\\nu_{1}(\bi) = b}}d_T(1,\bi)\bx^{\bi}\right)+2^{k-1}-1\\
    &= \left(\sum_{a=0}^{k-1}{k-1\choose a} \sum_{b=0}^{k-1-a}{k-1-a\choose b} x_{1}^b\sum_{\bi \in \{2,\ldots, n\}^{k-1-a-b}}d_T(1,\bi)\bx^{\bi}\right)+2^{k-1}-1\\
    &= \left(\sum_{a+b=0}^{k-1}{k-1\choose a,b,k-1-a-b}x_{1}^b\sum_{\bi \in \{2,\ldots, n\}^{k-1-a-b}}d_T(1,\bi)\bx^{\bi}\right)+2^{k-1}-1\\
    &= \left(\sum_{j=0}^{k-1}{k-1\choose j}\sum_{b=0}^j{j\choose b}x_{1}^b\sum_{\bi \in \{2,\ldots, n\}^{k-1-j}}d_T(1, \bi)\bx^{\bi}\right)+2^{k-1}-1\\
    &= \left(\sum_{j=0}^{k-1}{k-1\choose j}\left(x_{1}+1\right)^j\sum_{\bi \in \{2,\ldots, n\}^{k-1-j}}d_T(1,\bi)\bx^{\bi}\right)+2^{k-1} - 1,
    \end{align*} 
    where $j$ has been introduced to replace $a+b$. If $T$ is a tree on $2$ vertices, then $D_{1}p_T(\bx) = 2^{k-1}-1 \neq 0$. Therefore, the only Steiner nullvector of a tree on $2$ vertices is the trivial one, and the hyperdeterminant must be nonzero. We continue by assuming that $T$ is a tree on at least $3$ vertices.   

    Upon deleting the leaf vertex $0$, we obtain a new tree $T'$ on $n$ vertices. Let $\bx' = [x_1',x_2', \ldots, x_{n}']$ be a Steiner nullvector of $T'$. Notice that $\deg_T(1) = \deg_{T'}(1)-1$ and $\deg_T(2) = \deg_{T'}(2), \ldots,\deg_T(n) = \deg_{T'}(n).$ Using Proposition \ref{interior} and the definition of $\bx$, by setting $x'_v$ = $x_v$ for some leaf vertex common to both $T$ and $T'$, we conclude that $\bx' = [x_1+x_0, x_2, \ldots, x_n] = [x_1+1, x_2, \ldots, x_n]$. Then \begin{align*}
        D_{1}p_{T'}(\bx') &= \sum_{\bi\in \{1,\ldots, n\}^{k-1}}d_T(1,\bi)\bx'^{\bi} \\
        &= \sum_{j=0}^{k-1}\sum_{\substack{\bi\in \{1,\ldots, n\}^{k-1}\\ \nu_{1}(\bi) = j}}d_T(1,\bi)\bx'^{\bi} \\
        &= \sum_{j=0}^{k-1}{k-1\choose j}\left(x_{1}'\right)^j\sum_{\bi \in \{2,\ldots, n\}^{k-1-j}}d_T(1,\bi)\bx^{\bi}\\
        &= \sum_{j=0}^{k-1}{k-1\choose j}\left(x_{1}+1\right)^j\sum_{\bi \in \{2,\ldots, n\}^{k-1-j}}d_T(1,\bi)\bx^{\bi}.
    \end{align*}  
    Since we assumed $\bx$ and $\bx'$ were Steiner nullvectors, $D_{1}p_T(\bx) = 0 = D_{1}p_{T'}(\bx')$. Therefore, \begin{align*}
        0 &= D_{1}p_T(\bx) - D_{1}p_{T'}(\bx')\\
        &= \left(\sum_{j=0}^{k-1}{k-1\choose j}\left(x_{1}+1\right)^j\sum_{\bi \in \{2,\ldots, n\}^{k-1-j}}d_T(1,\bi)\bx^{\bi}\right)+2^{k-1} - 1 \\
        &\quad \quad \quad - \sum_{j=0}^{k-1}{k-1\choose j}\left(x_{1}+1\right)^j\sum_{\bi \in \{2,\ldots, n\}^{k-1-j}}d_T(1,\bi)\bx^{\bi}\\
        &= 2^{k-1}-1 \neq 0.
    \end{align*}
    Therefore, the only nullvector is the trivial one, and the hyperdeterminant must be nonzero. 
\end{proof}

\bibliographystyle{plain}
\bibliography{ref}

\end{document}